\documentclass[a4paper, 10pt, twoside]{article}

\usepackage{amsmath, amscd, amsfonts, amssymb, amsthm, latexsym, url, color, todonotes, bm, framed, rotating, enumerate} %pdflscape}
\usepackage{graphicx}
\usepackage[left=1in, right=1in, top=1.2in, bottom=1in, includefoot, headheight=13.6pt]{geometry}
\usepackage{mathtools}
\input{xy}
\xyoption{all}

\usepackage[colorlinks,breaklinks=true]{hyperref}
\usepackage[figure,table]{hypcap}
\hypersetup{
	bookmarksnumbered,
	pdfstartview={FitH},
	citecolor={black},%{blue},
	linkcolor={black},%{red},
	urlcolor={black},
	pdfpagemode={UseOutlines}
}
\makeatletter
\newcommand\org@hypertarget{}
\let\org@hypertarget\hypertarget
\renewcommand\hypertarget[2]{%
  \Hy@raisedlink{\org@hypertarget{#1}{}}#2%
} 
\makeatother 

\newtheorem{theorem}{Theorem}[section]
\newtheorem{lemma}[theorem]{Lemma}
\newtheorem{Key lemma}[theorem]{Key lemma}

\newtheorem{proposition}[theorem]{Proposition}

\theoremstyle{definition}

\newtheorem{definition}[theorem]{Definition}
\newtheorem{remark}[theorem]{Remark}
\newtheorem{example}[theorem]{Example}
\newtheorem{conjecture}[theorem]{Conjecture}

\renewcommand{\labelenumi}{(\roman{enumi})}

\newcommand{\xysquare}[8]{
\[\xymatrix{
#1 \ar@{#5}[r] \ar@{#6}[d] & #2 \ar@{#7}[d]\\
#3 \ar@{#8}[r] & #4
}\]
}

\newcommand{\bb}{\mathbb}

\newcommand{\comment}[1]{}

\renewcommand{\phi}{\varphi}

\newcommand{\roi}{\mathcal{O}}

\newcommand{\ul}[1]{\underline{#1}}

\newcommand{\xto}{\xrightarrow}

\renewcommand{\cal}{\mathcal}
\renewcommand{\hat}{\widehat}
\renewcommand{\frak}{\mathfrak}
\newcommand{\indlim}{\varinjlim}
\renewcommand{\tilde}{\widetilde}

\renewcommand{\projlim}{\varprojlim}

\DeclareMathOperator{\codim}{codim}

\DeclareMathOperator{\Frac}{Frac}

\DeclareMathOperator{\Hom}{Hom}

\DeclareMathOperator{\Pic}{Pic}

\DeclareMathOperator{\Spec}{Spec}

\DeclareMathOperator{\Sm}{Sm}
\DeclareMathOperator{\Spt}{Spt}

\newcommand{\CH}{C\!H}

\def\T{{\cal T}}
\def\et{{\mathrm{\acute{e}t}}} 
\newcommand{\K}{\hat{\cal K}}

\usepackage[bbgreekl]{mathbbol}
\DeclareSymbolFontAlphabet{\mathbbm}{bbold}

\usepackage{fancyhdr}

\usepackage{sectsty}
\sectionfont{\Large\sc\centering}
\chapterfont{\large\sc\centering}
\chaptertitlefont{\LARGE\centering}
\partfont{\centering}

\begin{document}
\itemsep0pt

\title{The Gersten conjecture for $p$-adic \'etale Tate twists and the $p$-adic cycle class map}

\author{Morten L\"uders}
%\address{Institut de Mathématiques de Jussieu–Paris Rive Gauche, UMPC - 4 place Jussieu, Case 247, 75252 Paris, France}
%\email{morten.lueders@ur.de}
%\classno{}

\date{}
%\extraline{}

\maketitle

\begin{abstract}
We prove the Gersten conjecture for $p$-adic \'etale Tate twists for a smooth scheme $X$ in mixed characteristic in the Nisnevich topology. Our main observation is that, while $p$-adic \'etale Tate twists are not $\bb A^1$-invariant, for the proof  of the Gersten conjecture it suffices that they satisfy the $\bb P^1$-bundle formula. This fits nicely with the emphasis on the projective bundle formula in non $\bb A^1$-invariant motivic cohomology recently developed by Elmanto-Morrow and Annala-Hoyois-Iwasa. Furthermore, identifying $p$-adic \'etale Tate twists with the syntomic cohomology defined by Bhatt-Morrow-Scholze, the result generalises the Gersten conjecture for logarithmic deRham-Witt sheaves due to Gros-Suwa to arbitrary characteristic. We also extend work of Schmidt-Strunk on the Gersten conjecture for presheaves of spectra to the non-$\bb A^1$-invariant situation. In the second part of the article, we revisit the cycle class map from thickened zero-cycles on the special fiber of $X$ to \'etale cohomology with coefficients in $p$-adic \'etale Tate twists previously studied in \cite{Lueders2019}. This cycle class map is important in the study of zero-cycles on smooth projective varieties over local fields and the approach to the cycle class map which we use in this article is more conceptual and, in contrast to the approach in \textit{loc. cit.}, works for arbitrary finite residue fields.
\end{abstract}

%\tableofcontents

\section{Introduction}
Let $k$ be a field and $X$ a smooth $k$-scheme of dimension $d$. Then Quillen \cite{Quillen1973} showed that the Gersten complex
$$0\rightarrow \mathcal{K}_{q,X}\xrightarrow{} \bigoplus_{x\in X^{(0)}}i_{x*} K_q(x)\rightarrow \bigoplus_{x\in X^{(1)}} i_{x*}K_{q-1}(x)\rightarrow...\to \bigoplus_{x\in X^{(d)}}i_{x*}K_{q-d}(x)\to 0 $$
for algebraic $K$-theory is exact in the Zariski topology, thereby verifying Gersten's conjecture for algebraic $K$-theory for smooth schemes over a field. This has had, among others, many applications to algebraic cycles because of the Bloch-Quillen formula $H^q(X,\mathcal{K}_{q,X})\cong \CH^q(X)$ and has motivated the study of the Gersten conjecture for other theories. For example, the Gersten conjecture may be formulated for any cohomology theory with supports $H^{*}(-): \mathrm{Sch}_S\to \rm Ab$ and base scheme $S$.\footnote{A cohomology theory with supports is, more precisely, a contravariant functor which to a pair of an $S$-scheme $X$ and a closed subscheme $Z\subset X$ associates a $\bb Z$-graded abelian group $H^{*}_Z(X)$, and which to any triple $Z\subset Y\subset X$ with $Z,Y$ closed in $X$ associates a long exact sequence
$$\dots \to H^*_Z(X)\to H^*_Y(X)\to H^*_{Y-Z}(X-Z)\to H^{*+1}_Z(X)\to \dots$$
These sequences give rise to the coniveau spectral sequence, whose rows are the Cousin complexes (see \cite[Sec. 5]{CHK97} or Section \ref{section_general_reductions}).}
\begin{conjecture}(Gersten {conjecture}) Let $\mathcal{H}^{q}_{\mathrm{Zar},X}$ be the Zariski sheafification of the presheaf $U\mapsto H^q(U)$. If $X$ is a regular scheme, then the Cousin complex
$$\bigoplus_{x\in X^{(0)}}i_{x*} H^{q}_x(X)\rightarrow \bigoplus_{x\in X^{(1)}} i_{x*}H^{q+1}_x(X)\rightarrow...\to \bigoplus_{x\in X^{(d)}}i_{x*}H^{q+d}_x(X)\to 0 $$
is a resolution of $\mathcal{H}^{q}_{\mathrm{Zar},X}$ in the Zariski topology.
\end{conjecture}
Examples over a field include $H^{q}(-)=H^{q}(-,\mu_\ell^n)$ with $(\ell,ch(k))=1$ and $H^{q}(-)=H^{q}(-,W_r\Omega^n_{X,\log})$ in which case $\mathcal{H}^{q}_{\mathrm{Zar},X}$ is isomorphic to $R^q\epsilon_* \mu_\ell^n$ and $R^q\epsilon_*W_r\Omega^n_{\log}$ respectively. Here $\epsilon:X_{\et}\to X_{\rm Zar}$ is the natural morphism of sites. In these cases, the cohomology of the respective Cousin complexes are the $E_2$-terms of the Leray spectral sequence associated to $\epsilon$. The main ingredient of Quillen's proof is a noether normalisation lemma to reduce the question to $\bb A^d_k$ and this idea is also used by Bloch-Ogus \cite{BO74} to show that the  Gersten conjecture holds for smooth schemes over a field for \'etale cohomology with coefficients prime to the characteristic of $k$, i.e. $H^{q}(-)=H^{q}(-,\mu_\ell^n)$ with $(\ell,ch(k))=1$. In the approach of Bloch-Ogus, however, difficult properties like Poincar\'e duality are required. Subsequently, Gabber developed a geometric presentation lemma and observed that this lemma and the $\bb P^1$-bundle formula suffice to prove the Gersten conjecture. This was most notably used by Gros-Suwa \cite{GrosSuwa1988} to show that the Gersten conjecture holds for $H^{q}(-)=H^{q}(-,W_n\Omega^n_{\log})$. The logarithmic deRham-Witt sheaves $W_n\Omega^n_{\log}$ are particularly important since they can be used to define $p$-adic \'etale motivic cohomology in characteristic $p$. Gabber's approach was axiomatised by Colliot-Th\'el\`ene-Hoobler-Kahn \cite{CHK97}. 
More recently, in the context of their motivic cohomology of equicharacteristic schemes, Elmanto-Morrow have shown that any presheaf on smooth $k$-schemes which satisfies the $\bb P^1$-bundle formula is deflatable and that any such Nisnevich presheaf which is deflatable satisfies the Gersten injectivity \cite[Lem. 6.11]{ElmantoMorrow} (see also Remark \ref{remark_Elmanto_Morrow}).

In this article, we focus on the mixed characteristic situation: let $A$ be a discrete valuation ring with local parameter $\pi$, perfect residue field $k$ of characteristic $p>0$ and fraction field $K$ of characteristic zero. Let $S=\Spec(A)$ and $X$ be a smooth $S$-scheme of dimension $d$. In this situation, the Gersten conjecture for algebraic K-theory with finite coefficients has been proved by Gillet-Levine \cite{GiLe87} and Geisser-Levine \cite{GeisserLevine2000} and for motivic cohomology with finite coefficients by Geisser \cite{Geisser2004}. The case of \'etale cohomology with torsion coefficients prime to $ch(k)$ has been studied by Schmidt-Strunk \cite{SS19} and the author \cite{Lu2021}. All of these cases have in common, that the theories are $\bb A^1$-invariant, at least for regular schemes. 
Our focus in this article will be on the Gersten conjecture for $p$-adic \'etale Tate twists which we denote by $\cal T_r(n)\in D^b(X_{\et},\bb Z/p^r)$. $p$-adic \'etale Tate twists were first defined by Schneider in the smooth case and then by Sato in the semi-stable case \cite{Sa07} and are generally considered to be the right $p$-adic \'etale motivic cohomology with mod $p^r$-coefficients in mixed characteristic for smooth and semi-stable schemes.
% and thereby closely related to $p$-adic algebraic K-theory. 
The first main result of this article is the following:
\begin{theorem}\label{theorem_intro}(Thm. \ref{theorem_main_in_text})
Assume that $A$ is henselian and let $\epsilon:X_{\et}\to X_{\rm Nis}$ be the natural morphism of sites. Then for all $n\geq 0$ the complex of sheaves
$$0\to R^q\epsilon_*\cal T_r(n) \to \bigoplus_{x\in X^{(0)}}i_{x*}H^{q}_x(X,\cal T_r(n))\to \bigoplus_{x\in X^{(1)}}i_{x*}H^{q+1}_x(X,\cal T_r(n))\to $$
$$\dots\to \bigoplus_{x\in X^{(d)}}i_{x*}H^{q+d}_x(X,\cal T_r(n))\to 0 $$
is exact in the Nisnevich topology.
\end{theorem}
The cohomology of $p$-adic \'etale Tate twists, just like the cohomology of logarithmic deRham-Witt sheaves, is not $\bb A^1$-invariant but satisfies the projective bundle formula (see Section \ref{section_proj_bun_formula}) and our key observation, going back to Gabber's mentioned above, is that the $\bb P^1$-bundle formula in combination with a geometric presentation lemma due to Schmidt-Strunk suffices to prove Theorem \ref{theorem_intro}. The theorem may therefore be considered to be an analogue of the main theorem of Gros-Suwa on the Gersten conjecture for logarithmic deRham-Witt sheaves.
Furthermore, both theories have found a common generalisation in the syntomic complexes $\bb Z/p^r(n)_X$ first defined by Bhatt-Morrow-Scholze \cite{BhattMorrowScholze2019} for quasi-syntomic rings and whose definition has been extended to arbitrary schemes by Bhatt-Lurie \cite{bhatt2022absolute}: for any regular $\bb F_p$-scheme $X$ there exist isomorphisms $\bb Z/p^r(n)_X\cong W_n\Omega^n_{X,\log}[-n]$ \cite[Sec. 8]{BhattMorrowScholze2019} and $\bb Z/p^r(n)_X\cong \T_r(n)\in D^b(X_{\et},\bb Z/p^r)$ for $X$ as in Theorem \ref{theorem_intro} \cite[Thm. 5.8] {BhattMathew2023}. In Theorem \ref{theorem_intro}, the reader may therefore replace $\T_r(n)$ by $\bb Z/p^r(n)_X$ and $A$ by a perfect field or a henselian dvr with perfect residue field. Due to the importance of syntomic cohomology, we first focus on this case, and then prove a generalisation which holds for Nisnevich local fibrant spectra following work of Schmidt-Strunk \cite{SS19}. Finally, we remark that these observations fit nicely with the emphasis on the projective bundle formula in non $\bb A^1$-invariant motivic cohomology recently developed by Elmanto-Morrow \cite{ElmantoMorrow} and Annala-Hoyois-Iwasa \cite{AHI2025}.

%This theorem complements other results on the Gersten conjecture for smooth schemes in mixed characteristic. Gillet-Levine prove it for $K$-theory \cite{GiLe87}, Geisser for motivic cohomology \cite{Geisser2004}, Schmidt-Strunk \cite{SS19} and the author \cite{Lu2021} for \'etale cohomology with finite coefficients away from the residue characteristic. For $q\leq n$ our theorem can also be deduced from Geisser's results for motivic cohomology and his results on the Beilinson-Lichtenbaum conjecture in mixed characteristic. The approach of this article just uses properties of $p$-adic \'etale Tate twists. The key observation is the following. One cannot use the formalism of a cohomology theory with duality combined with Quillen's trick as in \cite{BO74,Lu2021} since there is no Poincar\'e duality in all weights for $p$-adic \'etale Tate twists. Therefore, one has to use Gabber's approach \cite{Gabber1994,CHK97}, the main ingredients of which are a geometric presentation lemma, generalised to mixed characteristic by Schmidt-Strunk \cite{SS18}, and strict $\bb A^1$-invariance or the projective bundle formula. While $p$-adic \'etale Tate twists are not strictly $\bb A^1$-invariant, they satisfy the projective bundle formula.

Our second main theorem is the following:
\begin{theorem}(Thm. \ref{theorem_comparison})\label{theorem_2}
Let the notation be as above but let $X/S$ be smooth projective and $d$ be the relative dimension of $X$ over $S$. Assume that $k$ is finite. Let $X_k$ be the special fiber of $X$ (in the following also denoted by $X_1$) and denote the inclusion by $i:X_k\to X$. Let $X_n=X\times_A A/\pi^n$ and let $\hat{\cal K}^M$ denote the improved Milnor K-sheaf defined in \cite{Kerz2010}. 
Then for $q\in\{0,1\}$ there are isomorphisms
$$H^{d-q}_{\rm Nis}(X_k,i^*\hat{\cal K}^M_{d,X}/p^r) \xto{\cong}  \projlim_n H^{d-q}_{\rm Nis}(X_k,\hat{\cal K}^M_{d,X_n}/p^r)\xto{\cong} H^{2d-q}_\et(X_k,i^*\cal T_r(d)). $$
\end{theorem}
This was proved for $q=0$ and $A=W(k)$ and assuming that the $k$ is sufficiently large in \cite[Prop. 1.4]{Lueders2019} and the proof used the difficult techniques of \cite[Sec. 7]{BEK14} which only work if $ch(k)>d$. The proof we give in this article is more conceptual and works for arbitrary finite residue fields. The key ingredients are the calculation of the top cohomology of $\cal T_r(n)$, which also enters in the proof of Theorem \ref{theorem_intro}, the Kato conjectures for logarithmic deRham-Witt sheaves and the Gersten conjecture for $\hat{\cal K}^M_{d,X,\rm Nis}/p^r$ which was recently proved by Morrow and the author.
The importance of Theorem \ref{theorem_2} comes from the commutative diagram  
$$
\xymatrix{
 \CH^d(X_K)/p^r & \CH^d(X)/p^r\ar[r]^-{res^d_{/p^r}} \ar[d]_{cl_X} \ar@{->>}[l] & \projlim_n H^{d}_{\rm Nis}(X_k,\K^M_{d,X_n}/p^r) \ar[d]^{cl_{X_k}}  \\
& H^{2d}_\et(X, \cal T_r(d)) \ar[r]^-\cong & H^{2d}_\et(X_k,i^*\cal T_r(d)),   \\
}
$$
$K=\Frac(A)$, which is used in \cite{Lueders2019} and \cite{SS14} to study the Chow group of zero cycles mod $p^r$ of smooth projective schemes over local fields. Furthermore, $cl_{X_k}$ can be considered to be a thickened analogue of the isomorphism $\CH^d(X_k)/p^r\xto{\cong} H^{2d}(X,W_r\Omega^d_{r,\log}[-d])\cong \pi_1^{\rm ab}(X_k)/p^r$ which is studied in higher-dimensional unramified class field theory (mod $p^r$) over finite fields. 

\paragraph{Acknowledgement.} I would like to thank Tess Bouis for helpful comments and acknowledge support by the Deutsche Forschungsgemeinschaft (DFG, German Research Foundation) through the Collaborative Research Centre TRR 326 \textit{Geometry and Arithmetic of Uniformized Structures}, project number 444845124. Furthermore, I would like to thank the anonymous referees for very detailed and helpful reports.

\section{$p$-adic \'etale Tate twists}\label{section_etaletatetwists}
Let $A$ be a discrete valuation ring with perfect residue field of characteristic $p>0$ and fraction field $K$ of characteristic zero. Let $X$ be a regular scheme which is flat of finite type over $S=\Spec A$ and which is a smooth or semistable family over $S$. Let $Y$ be the special fiber and $X_K$ the generic fiber with inclusions
$$X_K\xto{j} X \xleftarrow{i} Y.$$
In \cite{Sa07}, and for all $n\geq 0$ and $r\geq 1$, Sato constructs a pair $(\T_r(n)_X,t')$, with $\T_r(n)_X\in D^b(X_{\et},\bb Z/p^r\bb Z)$ unique up to unique isomorphism, fitting into a distinguished triangle of the form
$$i_*\nu^{n-1}_{Y,r}[-n-1]\xto{} \T_r(n)_X\xto{t'}\tau_{\leq n}Rj_*\mu_{p^r}^{\otimes n}\xto{}i_*\nu^{n-1}_{Y,r}[-n].$$
The complexes $\T_r(n)_X$ are called \textit{$p$-adic \'etale Tate twists}.
If $X$ is smooth over $S$, then $\nu^{n-1}_{Y,r}=W_r\Omega^{n-1}_{Y,\log}$ and $\nu^{<0}_{Y,r}=0$. In the smooth case the objects $\T_r(n)_X$ were first defined by Schneider \cite{Schneider1994} and extensively studied by Geisser \cite{Geisser2004}.
{Quoting} \cite{Sa07}, we list the properties of $p$-adic \'etale Tate twists which we need for the Gersten conjecture.

\begin{enumerate}
\item[\bf T1]  
There is an isomorphism
       $t : j^*\T_r(n)_X \simeq \mu_{p^r}^{\otimes n}$ (see \cite[Def. 4.2.4]{Sa07}).
\item[\bf T2] %{\bf T2 (Acyclicity).}
$\T_r(n)_X$ is concentrated in $[0,n]$, i.e.,
   the $q$-th cohomology sheaf is zero unless $0 \leq q \leq n$ (see \cite[Def. 4.2.4]{Sa07}).
\item[\bf T3] %{\bf T3 (Purity).}
For a locally closed regular subscheme $i : Z \hookrightarrow X$ of codimension
          $c \,(\ge 1)$ which is contained in the special fiber $Y$, there is a Gysin isomorphism
$W_r\Omega_{Z,\log}^{n-c}[-n-c] \xto{\cong}\tau_{\leq n+c} R i^!\T_r(n)_X$ in $D^b(Z_{\et},\bb Z/p^r\bb Z)$ (see \cite[Thm. 4.4.7]{Sa07}).
\item[\bf T4]
Let $i_y: y \hookrightarrow X$ and $i_x: x\hookrightarrow X$ be points on $X$
    with $\mathrm{ch}(x)=p$, $x \in \overline{ \{ y \} }$ and $\mathrm{codim}_X(x)=\mathrm{codim}_X(y)+1$.
Put $c:=\codim_X(x)$. Then the connecting homomorphism
$$
\begin{CD}
%\delta_{y,x}^{\loc}(\T_r(n)_X):
    R^{n+c-1}i_{y*}(Ri_y^!\T_r(n)_X) @>>> R^{n+c}i_{x*}(Ri_x^!\T_r(n)_X)
\end{CD}
$$
in localization theory
agrees with the (sheafified) boundary map of
    Galois cohomology groups due to Kato (see \cite[Sec. 1.8]{Sa07}) 
$$
\begin{CD}
%\partial^{\val}_{y,x} :
    \left. \begin{cases}
     R^{n-c+1}i_{y*}\mu_{p^r}^{\otimes n-c+1} \quad& (\mathrm{ch}(y)=0)\\
     i_{y*} W_r\Omega^{n-c+1}_{y,\log}  \quad& (\mathrm{ch}(y)=p)
    \end{cases}
    \right\}
      @>>>
         i_{x*}W_r\Omega^{n-c}_{x,\log}
\end{CD}
$$
up to a sign depending only on $(\mathrm{ch}(y),c)$, via Gysin isomorphisms.
Here the Gysin map for $i_y$ with $\mathrm{ch}(y)=0$
   is defined by the isomorphism $t$ in {$\bf T1$} and Deligne's cycle class
     in $R^{2c-2}i_y^!\mu_{p^r}^{\otimes c-1}$.
\item[\bf T5] %{\bf T5 (Product structure).}
There is a unique morphism
$
\T_r(m)_X \otimes^{\bb L} \T_r(n)_X
      \to \T_r(m+n)_X$ in $D^b(X_{\et},\bb Z/p^r\bb Z)$
that extends the natural isomorphism
         $\mu_{p^r}^{\otimes m} \otimes \mu_{p^r}^{\otimes n}
           \simeq \mu_{p^r}^{\otimes m+n}$ on $X[1/p]$ (see \cite[Prop. 4.2.6]{Sa07}).
\item[\bf T6] Let $X'$ be another scheme which is flat of finite type over $S$
    and for which the objects $\T_r(n)_{X'}$ are defined.
Let $f : X' \to X$ be a morphism of schemes and let
    $\psi : X'[1/p] \to X[1/p]$ be the induced morphism. Then there is a unique morphism
$$
\begin{CD}
f^*\T_r(n)_X
       @>>> \T_r(n)_{X'} 
\end{CD}
$$
in $D^b(X'_{\et},\bb Z/p^r\bb Z)$ that extends the natural isomorphism
    $\psi^*\mu_{p^r}^{\otimes n} \simeq \mu_{p^r}^{\otimes n}$
     on $X'[1/p]$ (see \cite[Prop. 4.2.8]{Sa07}). We denote the induced map on cohomology by $$f^*:H^q(X,\T_r(n)_X)\to H^q(X',f^*\T_r(n)_X)\to H^q(X',\T_r(n)_{X'}).$$

\item[\bf T7] Assume now that $f$ is proper, and put
    $c:=\dim(X)-\dim(X')$. Then by \cite[Thm. 7.1.1]{Sa07} there is a unique morphism
$$
\begin{CD}
   Rf_*\T_r(n-c)_{X'}[-2c]
       @>>> \T_r(n)_X
\end{CD}
$$
in $D^b(X_{\et},\bb Z/p^r\bb Z)$ that extends the trace morphism
   $R\psi_*\mu_{p^r}^{\otimes n-c}[-2c] \to \mu_{p^r}^{\otimes n}$
     on $X[1/p]$. The latter trace morphism (on the generic fiber) was first defined by Deligne in \cite[XVIII, Thm. 2.9]{SGA4} under some additional assumptions. Using Gabber's absolute purity theorem \cite{Fu02}, Deligne's trace can be extended to proper morphisms.
   \end{enumerate}  
%\end{quote}
%\smallskip
%\noindent

\begin{proposition}\label{proposition_purity_tate_twists} Let $X$ be as above and $i : Z \hookrightarrow X$ a regular closed subscheme of codimension
          $c \,(\ge 1)$ which is contained in the special fiber $Y$. Then there are isomorphisms
$$H^{n+c}_Z(X,\T_r(n))\cong H^{n-c}(Z,W_r\Omega^{n-c}_{Z,\log}[-(n-c)])$$
for all $n\geq 0$ (by definition, $W_r\Omega^{q}_{Z,\log}$ is the zero sheaf for $q<0$).
%In particular, for $x\in X$ a point of characteristic $p$,
%$$H^{c+n}_x(X,\T_r(n))\cong H^{n-c}(x,W_r\Omega^{n-c}_{x,\log}[-(n-c)])=:  H^{n-c}(x,\bb Z/p^r(n-c)).$$
%\todo{replace $W_r$ by $v$.}
\end{proposition}
\begin{proof}
Consider the Grothendieck spectral sequence
$$E_2^{u,v}=H^{u}(Z,R^vi^!\T_r(n))\Longrightarrow R^{v+u}(\Gamma_Z\circ i^!)\T_r(n)=:H^{v+u}_Z(X,\T_r(n)).$$
By purity (property $\bf{T3}$ above), this spectral sequence looks as follows:
$$\xymatrix{
n+c+1  & H^{0}(Z,R^1i^!\T_r(n)) & \dots & & \\
n+c  & H^0(Z,W_r\Omega^{n-c}_{Z,\log}) \ar[drr] & H^1(Z,W_r\Omega^{n-c}_{Z,\log}) & \dots  &\\
n+c-1  & 0 & 0 & 0 & \dots
}$$
This implies that $H^0(Z,W_r\Omega^{n-c}_{Z,\log})\cong H^{n+c}_Z(X,\T_r(n))$. 

%For $x\in X$ as in the proposition, passing to the colimit over all neighbourhoods $U$ of $x$, we get that
%$$H^{c+n}_x(X,\T_r(n)):=\indlim_{U} H^{c+n}_{\overline{\{x\}}\cap U}(U,\T_r(n))\cong $$
%$$\indlim_{V\subset \overline{\{x\}}}H^{n-c}(V,W_r\Omega^{n-c}_{V,\log}[-(n-c)])=:H^{n-c}(x,W_r\Omega^{n-c}_{x,\log}[-(n-c)]).$$
\end{proof}

\subsection{The projective bundle formula}\label{section_proj_bun_formula}
For the projective bundle formula we need to extend the definition of $p$-adic \'etale Tate twists to negative weights.
\begin{definition}(\cite[Def. 3.5(i)]{Sa13})
Let the notation be as above. For $n<0$ and $r\geq 1$, we define $\T_r(n):=j_!\mathcal{H}om_{X_K} (\mu_{p^r}^{\otimes -n},\bb Z/p^r)[0]$ placed in degree zero in $D^b(X_{\et},\bb Z/p^r)$.
\end{definition}
Let $X$ be as in the beginning of the section and $E$ be a vector bundle of rank $a+1$ on $X$ and $f:\bb P(E)\to X$ be the associated projective bundle. Let $\roi(1)_E$ be the tautological invertible sheaf on $\bb P(E)$. By \cite[4.5.1]{Sa07} there is a distinguished triangle
$$\cal T_r(1)_{\bb P(E)}\to \roi^\times_{\bb P(E)}[0]\xto{\times p^r} \roi^\times_{\bb P(E)}[0]\to \cal T_r(1)_{\bb P(E)}[1]$$
in $D^b(\bb P(E)_{\et})$ and we denote the image of $\roi(1)_E$ under the induced map
$$H^1(\bb P(E),\roi^\times)\to H^2(\bb P(E),\cal T_r(1)_{\bb P(E)})$$
by $\xi$.
The composition
$$\cal T_r(n-q)_{X}[-2q]\xto{f^*} Rf_*\cal T_r(n-q)_{\bb P(E)}[-2q]\xto{-\cup \xi^q}Rf_*\cal T_r(n)_{\bb P(E)},$$
$0\leq q \leq a$, where $f^*$ is the adjoint of the morphism defined in {\textbf{T6}}, induces a canonical morphism
$$\gamma_E:\bigoplus_{q=0}^a \cal T_r(n-q)_{X}[-2q]\to Rf_*\cal T_r(n)_{\bb P(E)}$$
in $D^b(X_{\et},\bb Z/p^r)$.
\begin{proposition}\label{proposition_projective_bundle_formula} \cite[Thm. 4.1]{Sa13} \quad $\gamma_E$ is an isomorphism for all $n\in \bb Z$.
\end{proposition}
For a more general statement concerning syntomic cohomology see also \cite[Thm. 9.1.1]{bhatt2022absolute}. As mentioned in the introduction, the syntomic complexes of \textit{loc. cit.} coincide with Sato's $p$-adic \'etale Tate twists for schemes for which the latter are defined by \cite[Thm. 5.8]{BhattMathew2023}.

\section{General reductions for the Gersten conjecture}\label{section_general_reductions}
Let $A$ be a discrete valuation ring with local parameter $\pi$, perfect residue field $k$ of characteristic $p>0$ and fraction field $K$ of characteristic zero and $S=\Spec(A)$. 
Let $\cal S_S$ be the category of regular semistable schemes which are flat of finite type over $S=\Spec A$. Let  $\cal P_S$ be the category of pairs $(X,Z)$, where $X\in \cal S_S$ and $Z$ is a closed subset of $X$. A morphism $f:(X',Z')\to (X,Z)$ in $\cal P_S$ consists of a morphism $f:X'\to X$ in $\cal S_S$ such that $f^{-1}(Z)\subset Z'$.
We define a cohomology theory with supports by the contravariant functor
$$(X,Z)\mapsto H^*_Z(X,\cal T_r(n))$$
from $\cal P_S$ to the category of abelian groups.
For any triple $Z\subset Y\subset X$ with $Z,Y$ closed in $X$ there is a long exact sequence
$$\dots \to H^*_Z(X,\cal T_r(n))\to H^*_Y(X,\cal T_r(n))\to H^*_{Y-Z}(X-Z,\cal T_r(n))\to H^{*+1}_Z(X,\cal T_r(n))\to \dots$$
(\cite[Sec. 1.9]{Sa07}).
Furthermore, we have excision: let $Z\subset X$ and $Z'\subset X'$ be closed subschemes and $\pi:X'\to X$ be an \'etale morphism such that the restriction of $\pi$ to $Z'$ is an isomorphism and $\pi(X'-Z')\subset X-Z$, then 
$$H^*_Z(X,\cal T_r(n))\xto{\cong}H^*_{Z'}(X',\pi^*\cal T_r(n)).$$
In the following let $\tau\in \{\rm Zar,\rm Nis\}$ denote the Zariski or Nisnevich topology. Let $\cal H^*(\cal T_r(n))$ be the sheafification of the presheaf $U\mapsto H^*(U,\cal T_r(n))$ for the $\tau$-topology on $X$.
\begin{definition}
\renewcommand{\labelenumi}{(\arabic{enumi})}
\begin{enumerate}
\item Let $X\in \cal S_S$. We let
$$Z^p(X):=\{Z\subset X\; |\;Z \;\mathrm{ closed},\; \mathrm{codim}_X(Z)\geq p\}$$
for all $p\in \bb N$ and let $Z^p(X)=\emptyset$ for $q<0$. If $X$ is clear from the context, then we also write $Z^p$ instead of $Z^p(X)$. Let $Z^p/Z^{p+1}$ denote the ordered set of pairs $(Z, Z')\in Z^p\times Z^{p+1}$ such that $Z'\subset Z$, with the ordering $(Z,Z')\geq (Z_1,Z_1')$ if $Z\supset Z_1$ and $Z'\supset Z_1'$. The set $Z^p(X)$ should not be confused with the group of algebraic cycle, i.e. the free abelian group generated by integral closed subschemes of codimension $p$ of $X$, which is the standard notation. Since we do not use the latter group elsewhere in the article, we stick to the notation used by Bloch-Ogus in \cite[Sec. 3]{BO74}.
\item For the $\tau$-topology on $X$, we define $\cal H^q_{Z^p}(\cal T_r(n))$ to be the sheaf associated to the presheaf
$$U\mapsto \varinjlim_{Z\in Z^p(U)} H^q_{Z}(U,\cal T_r(n))$$
and $\cal H^q_{Z^p/Z^{p+1}}(\cal T_r(n))$ to be the sheaf associated to the presheaf
$$U\mapsto \varinjlim_{(Z,Z')\in Z^p(U)/Z^{p+1}(U)} H^q_{Z-Z'}(U-Z',\cal T_r(n)).$$
\item Let
%$$H_q(x,n):=\varinjlim_{U\subseteq \overline{\{x\}}} H_q(U,n),$$
%$$H^q(x,\cal T_r(n)):=\varinjlim_{U\subseteq \overline{\{x\}}} H^q(U,\cal T_r(n))$$
%and
$$H^q_x(X,\cal T_r(n)):=\varinjlim_{U\ni x} H^q_{\overline{\{x\}}\cap U}(U,\cal T_r(n)),$$
where the colimit is taken over all open subschemes 
%$U$ in $\overline{\{x\}}$ and all open subschemes 
$U$ in $X$ containing $x$.
\end{enumerate}
\end{definition}
\begin{proposition}\label{proposition_cousin_complex}
Let $\epsilon:X_{\et}\to X_{\tau}$ be the natural morphism of sites. If the natural map $$(*)\colon\cal H^q_{Z^{p+1}}(\cal T_r(n))\to \cal H^q_{Z^p}(\cal T_r(n))$$ is zero (in the $\tau$-topology) for all $p,q,n\in \mathbb{N}$, then there exists an exact sequence of $\tau$-sheaves   
$$
0\to R^q\epsilon_*\cal T_r(n) \to \bigoplus_{x\in X^{(0)}}i_{x*}H^{q}_x(X,\cal T_r(n))\to \bigoplus_{x\in X^{(1)}}i_{x*}H^{q+1}_x(X,\cal T_r(n))\to $$
 $$\dots\to \bigoplus_{x\in X^{(d)}}i_{x*}H^{q+d}_x(X,\cal T_r(n))\to 0.  
$$ 
\end{proposition}
\begin{proof}
Assuming $(*)=0$ for all $p,q,n\in \mathbb{N}$, the exact sequences 
$$\dots\to \cal H^{q-1}_{Z^p/Z^{p+1}}(\cal T_r(n))\to\cal H^q_{Z^{p+1}}(\cal T_r(n))\xto{0} \cal H^q_{Z^p}(\cal T_r(n))\to \cal H^{q}_{Z^p/Z^{p+1}}(\cal T_r(n))\to \dots$$
and
$$\dots\to \cal H^{q-1}_{Z^{p+1}/Z^{p+2}}(\cal T_r(n))\to\cal H^q_{Z^{p+2}}(\cal T_r(n))\xto{0} \cal H^q_{Z^{p+1}}(\cal T_r(n))\to \cal H^{q}_{Z^{p+1}/Z^{p+2}}(\cal T_r(n))\to \dots $$
splice together as follows:
$$\dots\to \cal H^{q-1}_{Z^p/Z^{p+1}}(\cal T_r(n))\twoheadrightarrow\cal H^q_{Z^{p+1}}(\cal T_r(n))\hookrightarrow \cal H^q_{Z^{p+1}/Z^{p+2}}(\cal T_r(n)) \twoheadrightarrow \cal H^q_{Z^{p+2}}(\cal T_r(n))\to\dots\; .$$
Note that the first two non-zero terms in the complex of the proposition coincide with $\cal H^q_{Z^{0}}(\cal T_r(n))$ and $\cal H^q_{Z^{0}/Z^{1}}(\cal T_r(n))$ and that $\cal H^{q-1}_{Z^{-1}/Z^{0}}(\cal T_r(n))=0$. It therefore remains to show that therefore there are isomorphisms 
$$\cal H^{q}_{Z^p/Z^{p+1}}(\cal T_r(n))\cong \bigoplus_{x\in X^{(q)}}i_{x*}H^{q}_x(X,\cal T_r(n)).$$ 
 This follows from Zariski excision: if $T_1,...,T_r$ are pairwise disjoint closed subsets of $X$, then $$\bigoplus_i H^q_{T_i}(X,\cal T_r(n))\cong H^q_{\cup T_i}(X,\cal T_r(n))$$ (see \cite[Lemma  1.2.1]{CHK97}). 
\end{proof}

\begin{remark} The complexes starting with $\bigoplus_{x\in X^{(0)}}$ in Proposition \ref{proposition_cousin_complex} are called \textit{Cousin complexes} and coincide by the construction in the proof with the rows of the coniveau spectral sequence (see \cite[Sec. 1.1]{CHK97} or \cite[Ch. IV]{Ha66}).
\end{remark} 

The following definition is used in applications to show that $(*)$ is zero.
\begin{definition}
Let $X\in \cal S_S$ and $x\in X$ a point. We say that $\cal T_r(n)$ is effaceable at $x$ if the following condition is satisfied: Given $p\geq 0$, for any affine $\tau$-neighbourhood $W$ of $x$ and any closed subscheme $Z\in Z^{p+1}(W)$ there exists a $\tau$-neighbourhood $f:U\to W$ of $x$ and a closed subscheme $Z'\in Z^p(U)$ containing $f^{-1}(Z)$ and $x$ such that the composition 
$$H^q_{Z}(W,\cal T_r(n))\to H^q_{f^{-1}(Z)}(U,\cal T_r(n))\to H^q_{Z'}(U,\cal T_r(n))$$
is zero for all $q$.
\end{definition}

\section{Proof of the Gersten conjecture}
We fix the following notation. Let $A$ be a discrete valuation ring with local parameter $\pi$, perfect residue field $k$ of characteristic $p>0$ and fraction field $K$ of characteristic zero, and $S=\Spec(A)$. 
%Let $A$ be a discrete valuation ring with perfect residue field $k$ of characteristic $p>0$ and fraction field $K$ of characteristic zero. Let $S=\Spec(A)$. 

In order to prove our main theorem, the Gersten conjecture for $p$-adic \'etale Tate twists, we closely follow Gabber's strategy \cite{Gabber1994} as exposed in \cite{CHK97}. A key ingredient is the following lemma. In equal characteristic, it has appeared in \cite[Key lemma, Sec. 4.1]{CHK97}, \cite[p. 621]{GrosSuwa1988} and, more recently, in \cite[Lem. 6.12]{ElmantoMorrow} (see also Remark \ref{remark_Elmanto_Morrow}). 
\begin{lemma}\label{prop_commutativity}
Let $V$ be a smooth $S$-scheme, $p:\bb A^1_V\to V$ and $\tilde p:\bb P^1_V\to V$ the natural projections. Let $j: \bb A^1_V\to \bb P^1_V$ be the inclusion and $s_\infty:V\to \bb P^1_V$ the section of $\tilde p$ at infinity. Let $F$ be a closed subset of $V$. Then the diagram
$$\xymatrix{
 H^{q}_{\bb A^1_F}(\bb A^1_V,\cal T_r(n))  &  \\
 & H^{q}_F(V ,\cal T_r(n)) \ar[ul]_{p^*} & \\
  H^{q}_{\bb P^1_F}(\bb P^1_V,\cal T_r(n)) \ar[uu]^{j^*}  \ar[ur]^{s_\infty^*} & 
}$$
commutes.
\end{lemma}
\begin{proof}
By Proposition \ref{proposition_projective_bundle_formula} there is an isomorphism $H^{q}_{\bb P^1_F}(\bb P^1_V,\cal T_r(n))\cong H^{q-2}_F(V ,\cal T_r(n-1))\oplus  H^{q}_F(V ,\cal T_r(n))$. Since the restrictions $j^* \roi(1)$ and $s_\infty^* \roi(1)$ are trivial, this implies that $j^*|_{H^{q-2}_F(V ,\cal T_r(n-1))}=0$ and $s_\infty^*|_{H^{q-2}_F(V ,\cal T_r(n-1))}=0$ and therefore that the diagram commutes.
\end{proof}

\begin{proposition}\label{proposition_A1_effacement}
Let $V$ be a smooth $S$-scheme, $F$ a closed subset of $V$ and $F'$ a closed subset of $\bb A^1_F$ such that the projection $f:F'\to F$ is finite. Diagrammatically:
$$\xymatrix{
 F' \ar@{^{(}->}[r] \ar[rd]_f   & \bb A^1_F \ar[d] \ar@{^{(}->}[r]  & \bb A^1_V \ar[d]^p  \\
 & F \ar@{^{(}->}[r]  & V. 
}$$
Then the natural map 
$$ H^{q}_{F'}(\bb A^1_V,\cal T_r(n))\to H^{q}_{\bb A^1_F}(\bb A^1_V,\cal T_r(n)) $$
is zero.
\end{proposition}
\begin{proof} The argument is essentially the  same as in the proof of \cite[Thm. 4.2.1]{CHK97}.
Let $k:\bb P^1_V-F'\hookrightarrow \bb P^1_V$ be the open immersion. Since $s_\infty(V)\cap F'=\emptyset$ as $f$ is finite, $s_\infty$ can be factorised into $s_\infty=k\circ s'$.
The statement now follows from the following commutative diagram in which the isomorphism on the left is due to excision, the middle triangle commutes by Lemma \ref{prop_commutativity} and the lower horizontal composition is part of the localisation exact sequence and therefore zero.

$$\xymatrix{
 H^{q}_{F'}(\bb A^1_V,\cal T_r(n)) \ar[r]   & H^{q}_{\bb A^1_F}(\bb A^1_V,\cal T_r(n))  & & & \\
 &  \ar[u] & H^{q}_F(V ,\cal T_r(n)) \ar[ul]_{p^*} & & \\
H^{q}_{F'}(\bb P^1_V,\cal T_r(n))\ar[r] \ar[uu]_\cong  & H^{q}_{\bb P^1_F}(\bb P^1_V,\cal T_r(n)) \ar[uu]^{j^*} \ar[rr]^-{k^*} \ar[ur]^{s_\infty^*} & & H^{q}_{\bb P^1_F-F'}(\bb P^1_V-F' ,\cal T_r(n)). \ar[ul]_{s'^*} 
}$$
\end{proof}

\begin{theorem}\label{theorem_main_in_text}
Assume that $S$ is henselian. Let $X/S$ be a smooth scheme of dimension $d$ and let $\epsilon:X_{\et}\to X_{\rm Nis}$ be the natural morphism of sites.
The complex of sheaves
$$0\to R^q\epsilon_*\cal T_r(n) \to \bigoplus_{x\in X^{(0)}}i_{x*}H^{q}_x(X,\cal T_r(n))\to \bigoplus_{x\in X^{(1)}}i_{x*}H^{q+1}_x(X,\cal T_r(n))\to $$
$$\dots\to \bigoplus_{x\in X^{(d)}}i_{x*}H^{q+d}_x(X,\cal T_r(n))\to 0 $$
is exact in the Nisnevich topology for all $n\geq 0$.
\end{theorem}
\begin{proof}
For the following argument with torsion coefficients of order prime to $p$ see \cite[Cor. 3.5]{Lu2021} and \cite[Sec. 5,6]{SS19}. We need to show that the natural map $$(*)\colon\cal H^q_{Z^{p+1}}(\cal T_r(n))\to \cal H^q_{Z^p}(\cal T_r(n))$$ is zero in the Nisnevich topology for all $p,q,n\in \mathbb{N}$. We show $(*)=0$ on stalks. Let $x\in X$ be a point contained in the special fiber. If $x$ is contained in the generic fiber, then after localising to the generic fiber we are in the case of a variety over a field and the statement follows from \cite{BO74} and \cite{CHK97}. We first assume that $k$ is infinite.
Let $W$ be an affine Nisnevich neighbourhood of $x$ and $Z\in Z^{p+1}(W)$ and assume that $Z_k\neq W_k$, $Z_k$ and $W_k$ being the special fibers of $Z$ and $W$ respectively. Then by \cite[Thm. 2.4]{SS18} we can find a Nisnevich neighbourhood $f:U\to W$ of $x$ and an \'etale morphism
$$\rho=(\psi,\nu): U\to V\times\bb A^1_S\subset \bb A^{d-1}_S\times \bb A^1_S,$$
where $V$ is an open subset of $\bb A^{d-1}_S$, with the following properties: Let $Z':=f^{-1}(Z)$, $F:=\psi(Z)$, $Z'':=\psi^{-1}(F)$ and $F':= \rho(Z')$. Then the map $Z'\to V$ is finite and $(U,Z')$ is a Nisnevich neighbourhood of $(\bb A^1_V,F')$. Therefore the left vertical morphism in the commutative diagram
$$\xymatrix{
 H^{q}_{Z'}(U,\cal T_r(n)) \ar[r]^{}  & H^{q}_{Z''}(U,\cal T_r(n))  \\
 H^{q}_{F'}(\bb A^1_V,\cal T_r(n)) \ar[r] \ar[u]_\cong  & H^{q}_{\bb A^1_F}(\bb A^1_V,\cal T_r(n)) \ar[u]  
}$$
is an isomorphism by excision and the lower horizontal morphism is zero by Proposition \ref{proposition_A1_effacement}. This implies that the composition 
$$H^*_{Z}(W,\cal T_r(n))\to H^*_{Z'}(U,\cal T_r(n))\to H^*_{Z''}(U,\cal T_r(n))$$
is zero.

We therefore get that $(*)\colon\cal H^q_{Z^{p+1}}(\cal T_r(n))\to \cal H^q_{Z^p}(\cal T_r(n))$ is zero for all $p>0$ and it remains to show that $$(*)\colon\cal H^q_{Z^{1}}(\cal T_r(n))\to \cal H^q_{Z^0}(\cal T_r(n))$$ is zero. 
First note that it suffices to show the statement for $q\leq n+1$. Indeed, by definition $\cal H^q_{Z^0}(\cal T_r(n))_x\cong H^q_{}(\roi_{X,x}^h,\cal T_r(n))$ and $$H^{q}(\roi^h_{X,x},\cal T_r(n))=0$$ for $q>n+1$ since by Gabber's affine analogue of proper base change \cite[Thm. 1]{GabberAffineAnalogue} there is an isomorphism $H^{q}(\roi^h_{X,x},\cal T_r(n))\cong H^{q}(\roi^h_{X,x}/\pi,i^*\cal T_r(n))$ for $i:\Spec (\roi^h_{X,x}/\pi)\to \Spec (\roi^h_{X,x})$ the inclusion. But $H^{q}(\roi^h_{X,x}/\pi,i^*\cal T_r(n))=0$ for $q>n+1$ since $\cal T_r(n)$ is concentrated in $[0,n]$ and the $p$-cohomological dimension of an affine noetherian scheme of characteristic $p$ is $\leq 1$ (see \cite[Exp. X, Thm. 5.1]{SGA4}).
Let us therefore assume that $q\leq n+1$. The map $(*)$ fits into the exact sequence $\cal H^q_{Z^{1}}(\cal T_r(n))\to \cal H^q_{Z^0}(\cal T_r(n))\to \cal H^q_{Z^0/Z^1}(\cal T_r(n))$. By definition $\cal H^q_{Z^0}(\cal T_r(n))_x\cong H^q_{}(\roi_{X,x}^h,\cal T_r(n))$ and $\cal H^q_{Z^0/Z^1}(n)_x\cong H^q_{}(K(\roi^h_{X,x}),\cal T_r(n))$ and we may assume that $x\in Y$. We can therefore also show that the composition  $ H^q_{}(\roi_{X,x}^h,\cal T_r(n))\to H^q_{}((\roi_{X,x}^h)_{\eta},\cal T_r(n)) \to H^q_{}(K(\roi_{X,x}^h),\cal T_r(n))$ is injective, where $\eta$ is the generic point of the special fiber of $\Spec(\roi_{X,x}^h)$. The map $ H^q_{}(\roi_{X,x}^h,\cal T_r(n))\to H^q_{}((\roi_{X,x}^h)_{\eta},\cal T_r(n))$ is obtained by localising the exact sequence $\cal H^q_{Z^1-\eta}(\cal T_r(n))\to \cal H^q_{Z^0}(\cal T_r(n)) \to \cal H^q_{Z^0/(Z^1-\eta)}(\cal T_r(n))$ at $x$ and by what we showed above, the first map in this sequence of sheaves is zero.  
The map $H^q_{}((\roi_{X,x}^h)_{\eta},\cal T_r(n)) \to H^q_{}(K(\roi_{X,x}^h),\cal T_r(n))$ is injective by the following proposition. 

%The reduction of the case in which the residue field of $S$ is finite to the above case can be achieved by a standard norm argument. For this we refer to \cite[Proof of Thm. 2.2.7]{CHK97} and \cite[Proof of Prop. 1.9]{LuMo20}.

Finally, it follows from a norm argument
%, along the lines of the proof of the isomorphism $G(A)\to \hat{G}(A)$ for $G\in \mathfrak{CT}$ in the proof of \cite[Prop. 9]{Ke10}, 
that we can drop the assumption that the residue field of $A$ is infinite. 
We give the argument for this reduction for the injectivity following \cite[Lem. 1.9]{LuMo20}; the exactness at the other places can be shown similarly.
We only need to worry about the case that $A$ has finite residue field $k$. Let $R:=\roi_{X,x}^h$ for $x$ a point of the special fiber and let $L$ denote its residue field. Since $k$ is perfect, $L$ is a finitely generated, separable field extension of $k$ and we may therefore realise $L$ as a finite separable extension of a rational function field $k(\ul t):=k(t_1,\dots,t_d)$.
We fix a prime $p$ which is coprime to $|L:k(\ul t)|$, and let $k^{(r)}$ be the unique degree $p^r$ extension of the finite field $k$. Note that $L\otimes_kk^{(r)}=L\otimes_{k(\ul t)} k^{(r)}(\ul t)$ is the tensor product of finite field extensions of coprime degree, hence is a field.
Let $A^{(r)}$ be the finite \'etale extension of $A$ corresponding to the extension $k^{(r)}$ of $k$, and set $R^{(r)}:=R\otimes_{A}A^{(r)}$. Note that $R^{(r)}$ is local since $R^{(r)}/\mathfrak m_RR^{(r)}=L\otimes_kk^{(r)}$ is a field, and henselian since it is finite over $R$. 
%By the same argument, we may additionally assume that, denoting $R_\eta:= \roi_{X,\eta}$, the ring $R_\eta^{(r)}:= R_\eta\otimes_{A}A^{(r)}$ is local and essentially smooth.  
Consider the commutative diagram
\[
  \xymatrix{
  R^q\epsilon_*\cal T_r(n)(R^{(r)}) \ar[r]^-{} \ar@{-->}@/^5mm/[d]^N  &  R^q\epsilon_*\cal T_r(n)(\Frac(R^{(r)}))   \ar@{-->}@/^5mm/[d]^N 
  \\
R^q\epsilon_*\cal T_r(n)(R) \ar[u]  \ar[r] & \ar[u]_{} R^q\epsilon_*\cal T_r(n)(\Frac(R))  
  }
\]
in which the dotted arrows correspond to the norm map $N$ (see for example \cite[Lem. 3.4]{LuMo20} or construct the norm map using the distinguished triangle in Section \ref{section_etaletatetwists} defining $p$-adic \'etale Tate twists). Let $a$ be in the kernel of the lower horizontal map. For primes $p$ and $p'$ as above with $(p,p')=1$, by a filtered colimit argument we get that for $r$ large enough $p^ra=0=p'^ra$ and therefore $a=0$.
\end{proof}

\begin{proposition}\label{prop_effaceability} 
Let $X$ and $S$ be as above and let $x\in X$ a point of the special fiber. Then the canonical map
$$H^{q}((\roi^h_{X,x})_{\eta},\cal T_r(n)) \to   H^{q}(K(\roi^h_{X,x}),\cal T_r(n))$$
is injective for all $n\geq 0$ and $n+1\geq q\geq 0$.
\end{proposition}
\begin{proof}
%Let $K=\Frac(\roi^h_{X,x})$.
Let $Y_x=\Spec((\roi^h_{X,x})_{\eta}/\pi)$ and $i:Y_x\to \Spec((\roi^h_{X,x})_{\eta})$ be the inclusion. We consider the long exact localisation sequence
$$\xymatrix{ 
\dots \quad \ar[r]& H^{n-1}_{Y_x}((\roi^h_{X,x})_{\eta},\cal T_r(n))  \ar[r] & H^{n-1}((\roi^h_{X,x})_{\eta},\cal T_r(n)) \ar[r] &  H^{n-1}(K(\roi^h_{X,x}),\cal T_r(n))\ar[r] &  & \\
& H^{n}_{Y_x}((\roi^h_{X,x})_{\eta},\cal T_r(n)) \ar[r] &  H^{n}((\roi^h_{X,x})_{\eta},\cal T_r(n)) \ar[r] &  H^{n}(K(\roi^h_{X,x}),\cal T_r(n))\ar[r] & & \\
& H^{n+1}_{Y_x}((\roi^h_{X,x})_{\eta},\cal T_r(n)) \ar[r] &  H^{n+1}((\roi^h_{X,x})_{\eta},\cal T_r(n)) \ar[r] &  H^{n+1}(K(\roi^h_{X,x}),\cal T_r(n))\ar[r] & \quad \dots &
}$$

First note that $H^{m}_{Y_x}((\roi^h_{X,x})_{\eta},\cal T_r(n))=0$ for $m\leq n$ since by definition the natural map $$H^{m}_{}((\roi^h_{X,x})_{\eta},\cal T_r(n))\to H^{m}_{}(K(\roi^h_{X,x}),\cal T_r(n))$$ is an isomorphism for $m<n$ and $$H^{n}((\roi^h_{X,x})_{\eta},\cal T_r(n)) \to   H^{n}(K(\roi^h_{X,x}),\cal T_r(n))$$ is injective. For the remaining case we show that the map $H^{n}(K(\roi^h_{X,x}),\cal T_r(n)) \to  H^{n+1}_{Y_x}((\roi^h_{X,x})_{\eta},\cal T_r(n))$ is surjective. By Proposition \ref{proposition_purity_tate_twists}, and since $c=\mathrm{codim}_{\Spec((\roi^h_{X,x})_{\eta})}(Y_x)=1$, we have that 
$$H^{n+1}_{Y_x}((\roi^h_{X,x})_{\eta},\cal T_r(n))\cong W_r\Omega^{n-1}_{Y_x,\log}.$$
Finally, there is a commutative diagram
$$\xymatrix{
 K^M_{n}(K(\roi^h_{X,x}))/p^r \ar[r]^-{\partial^M_\pi} \ar[d]^\cong  & K^M_{n-1}((\roi^h_{X,x})_{\eta}/\pi)/p^r \ar[r] \ar[d]^\cong  &0 \\
 H^{n}(K(\roi^h_{X,x}),\mu_{p^r}^n) \ar[r] \ar[d]^\cong & W_r\Omega^{n-1}_{Y_x,\log} \ar[d]^\cong & \\
H^{n}(K(\roi^h_{X,x}),\cal T_r(n)) \ar[r] & H^{n+1}_{Y_x}((\roi^h_{X,x})_{\eta},\cal T_r(n))  & 
}$$
in which the vertical maps are isomorphisms by the Bloch-Kato conjecture and T1 and the tame symbol $\partial^M_\pi$ is surjective. The commutativity of the upper square holds by the definition of $H^{n}(K(\roi^h_{X,x}),\mu_{p^r}^n) \to W_r\Omega^{n-1}_{Y_x,\log}$ in terms of symbols (see \cite[Sec. 1]{Ka86}). This map is also called Kato's residue map. The commutativity of the lower square follows from the compatibility of Kato's residue map with Gysin maps which holds by \textbf{T4}. For the last argument see also \cite[Proof of Lem. 1.3.1]{Sa07}.
%implying that $H^{n+1}(\roi^h_{X,x},\cal T_r(n)) \to  H^{n+1}(\roi^h_{X,x}[\frac{1}{\pi}],\cal T_r(n))$ is injective. 
\end{proof}

In the remainder of this section we formulate a generalisation of Theorem \ref{theorem_main_in_text} for presheaves of spectra following work of Schmidt-Strunk \cite{SS19}.
In \cite{SS19} Schmidt-Strunk prove the existence of a Gersten resolution for $\bb A^1$-Nisnevich local fibrant spectra which one would also like to have for non-$\bb A^1$-invariant cohomology theories. We quickly recall their notation and the results we need and then give a proof of the generalisation. 

For a noetherian finite dimensional base scheme $S$ we denote by $\Sm_S$ the category of smooth schemes of finite type over $S$, by $\Spt_{S^1}(\Sm_S)$ the category of presheaves of spectra on $\Sm_S$ and by $\Spt_{S^1}(X_{\mathrm{Nis}})$ the category of presheaves of spectra on the small Nisnevich site of $X{\in \Sm_S}$. We consider both categories with the (stable) object-wise model structure and their homotopy  categories are triangulated categories with exact triangles given by the homotopy (co)fiber sequences. The left Bousfield localization at the equivalences on Nisnevich stalks is called the Nisnevich local model structure.

The categories $\Spt_{S^1}(\Sm_S)$ and $\Spt_{S^1}(X_{\mathrm{Nis}})$ are related via the canonical covering preserving inclusion functor $\delta_{X/{S}}\colon X_{\mathrm{Nis}}\to \Sm_S$, precomposing with which gives the right adjoint of a Quillen adjunction
$$\delta_{X/{S}}^*\colon \Spt_{S^1}(X_{\mathrm{Nis}}) \rightleftarrows\Spt_{S^1}(\Sm_S) \colon \delta_{X/{S},*} $$
for the object-wise and the Nisnevich local model structures.
For more details we refer to the preliminary section of \cite{SS18} and \cite[Sec. 2.1]{SS19}. 

\begin{definition}
\begin{enumerate}
\item For $E\in \Spt_{S^1}(\Sm_S)$ or $E\in \Spt_{S^1}(X_{\mathrm{Nis}})$ and $q\in \bb Z$ we let $\pi_q(E)$ denote the presheaf of abelian groups $U\mapsto \pi_q(E(U))$ on $\Sm_S$ or $X_{\mathrm{Nis}}$ and set $E^q(X):=\pi_{-q}(E(X))$.
\item Let  $X\in \Sm_S$, $Z\subset X$ a closed subset with open complement $j:X\setminus Z\hookrightarrow X$. We denote
$$E_X:=\delta_{X/S,*}E\in \Spt_{S^1}(X_{\mathrm{Nis}})$$
and denote the homotopy fiber of the canonical morphism (see \cite[Lem. 2.1]{SS19})
$$\eta_j:E_X\to j_*j^*E_X=j_*E_{{X\setminus Z}}$$
by $E_{Z/X}$. More generally, for a morphism $f:\tilde X\to X\in \Sm_S$ we denote by 
$$\eta_f:E_X\to f_*E_{\tilde{X}}$$
the morphism given on an \'etale morphism of finite type $V\to X$ as the map $E_X(V)=E(V)\to E(V\times_X\tilde{X})=E_{\tilde{X}}(V\times_X\tilde{X})$ induced by the projection. We also denote by $\eta_f$ the morphism on homotopy fibers induced by a morphism of pairs  $f$ (see \cite[Constr. 5.3]{SS19}).
\item For a point $x\in X$ we denote the corresponding morphism by $\frak{z}\colon \Spec(\roi_{X,x})\to X$ and the canonical adjunction (see \cite[(2a)]{SS19}) by $$\frak{z}^*\colon \Spt_{S^1}(X_{\mathrm{Nis}}) \rightleftarrows\Spt_{S^1}(\Spec(\roi_{X,x})_{\mathrm{Nis}}) \colon \frak{z}_*. $$
\item Denote $$E_{X^{(s)}}:=\varinjlim_{Z^s(X)}E_{Z/X}$$
and the homotopy cofiber of $E_{X^{(s+1)}}\to E_{X^{(s)}}$ by $E_{X^{(s/s+1)}}$.
\end{enumerate}
\end{definition}

By \cite[Prop. 3.20, Cor. 3.21]{SS19} for every Nisnevich local fibrant spectrum $E\in \Spt_{S^1}(\Sm_S)$, there is an isomorphism
$$E^q_{X^{(s/s+1)}}\cong \bigoplus_{x\in X^{(s)}}\frak{z}_*\frak{z}^*E^q_{Z/X}$$
for all integers $q$, where we set $Z:=\overline{\{z\}}$ in each summand.
Then it is shown in
\cite[Sec. 4.2]{SS19} that for a Nisnevich local fibrant spectrum $E\in \Spt_{S^1}(\Sm_S)$ the coniveau filtration on $X$ and cofiber sequences for $E$ imply the existence of a \textit{Nisnevich Gersten complex} of sheaves on $X_{\mathrm{Nis}}$
$$0\to (E^q_X)^\sim\to \bigoplus_{x\in X^{(0)}}\frak{z}_*\frak{z}^*E_{Z/X}^{q}\xto{d^0} \bigoplus_{x\in X^{(1)}}\frak{z}_*\frak{z}^*E_{Z/X}^{q+1}\xto{d^1}\dots\xto{d^{d-1}} \bigoplus_{x\in X^{(d)}}\frak{z}_*\frak{z}^*E_{Z/X}^{q+d}\to 0,$$
where $(-)^\sim$ denotes the Nisnevich sheafification. The other terms are sheaves by \cite[Cor. 3.24]{SS19}.
Furthermore, 
it is shown in \cite[Lem. 4.1]{SS19} that the above complex is exact at the first spot $(E^q_X)^\sim$ if the canonical map 
\begin{equation}\label{eqn.exactness1} 
(E^q_{X^{(1)}})^\sim\to (E^q_{X^{(0)}})^\sim
\end{equation}  
is zero and it is exact at the spot $q+s$ if the canonical maps
\begin{equation}\label{eqn.exactness2}
(E_{X^{(s)}}^{q+s})^\sim\to (E_{X^{(s-1)}}^{q+s})^\sim\text{ and }(E_{X^{(s+2)}}^{q+s+1})^\sim\to (E_{X^{(s+1)}}^{q+s+1})^\sim
\end{equation} 
are both zero, where the first condition is empty if $s=0$. 

The following definition is due Colliot-Th\'el\`ene-Hoobler-Kahn, where it is stated as axiom $\textbf{SUB2}$. In work of Elmanto-Morrow \cite[Def. 6.7, Rem. 6.9]{ElmantoMorrow} it is lifted to a property of presheaves of spectra which are then called \textit{deflatable}.
\begin{definition}(\cite[Sec. 5.1]{CHK97})\label{remark_Elmanto_Morrow}
%In the context of their motivic cohomology of equicharacteristic schemes, Elmanto-Morrow have introduced the term \textit{deflatable}. We first recall the definition, replacing the base field by an arbitrary base scheme $S$.
Let $E\in \Spt_{S^1}(\Sm_S)$ and $V\in \Sm_S$. Let

$$
\xymatrix{
\bb A^1_V \ar[d]_j \ar[r]^\pi & V \\
\bb P^1_V \ar[ru]_{\bar{\pi}}          & }
$$ be the factorisation of the projection $\pi$ through the canonical open immersion $j$. Let $\bar{s}_\infty:V\to \bb P^1_V$ be the section at infinity. Then $E$ satisfies $\textbf{SUB2}$ if for all $V\in \Sm_S$ and all closed subsets $F$ of $V$
%is called \textit{deflatable} 
the diagram
$$\xymatrix{
\bar{\pi}_*j_*E_{\bb A^1_F/\bb A^1_V}={\pi}_*E_{\bb A^1_F/\bb A^1_V} & &  \\
  \ar[u] &  \bar{\pi}_*\bar{s}_{\infty,*} E_{F/V}=E_{F/V} \ar[ul]_-{\eta_\pi} &  \\
 \bar{\pi}_* E_{\bb P^1_F/\bb P^1_V} \ar[uu]^{ \bar{\pi}_*\eta_j}  \ar[ur]_{ \bar{\pi}_*\eta_{\bar{s}_\infty}} & 
}$$
commutes in the homotopy category.

%For $X\in {\rm Sm}_S$ we write $E^X$ for the presheaf $U\mapsto E(U\times_SX)$. There are morphisms of presheaves $$j^*,p^*s_\infty^*:E^{\bb P^1}\to E^{\bb A^1},$$
%induced by the following morphisms of schemes: 
%\begin{enumerate}
%\item[1.] $p$ is the projection $\bb A^1_S\to S$,
%\item[2.] $s_\infty$ is the closed immersion $S\to \bb P^1_S$ at infinity,
%\item[3.] $j: \bb A^1_S\to\bb P^1_S$ is the open %immersion complementary to the immersion in 2.
%\end{enumerate}
%Then $E$ is called \textit{deflatable} if the maps $j^*$ and $p^*s_\infty^*$ are homotopic.
\end{definition}

\begin{example}
\begin{enumerate}
\item It is shown in \cite[Lem. 5.7, Lem. 5.8]{SS19} that an $\bb A^1$(-Nisnevich) local fibrant spectrum satisfies $\textbf{SUB2}$.
\item Let $E(*)\in \mathrm{Fun}(\bb Z,\Spt_{S^1}(\Sm_S))$ be a graded presheaf of spectra equipped with a natural map
$$c_1\colon \Pic X\to \Hom(E_X(n-1),E_X(n))$$
for all $X\in \Sm_S$ and $n$. 
Suppose furthermore, that $E(*)$ satisfies the $\bb P^1$-bundle formula, i.e. that the natural map
$$c_1(\roi(1))\circ \pi^*\oplus \pi^*\colon E_X(n-1)\oplus E_X(n)\to E_{\bb P^1_X}(n)$$
is a homotopy equivalence for all $X\in \Sm_S$ and $n$. Then $E(n)$ satisfies $\textbf{SUB2}$ for all $n$. Indeed, the commutativity of the above diagram can be checked immediately on both summands coming from the $\bb P^1$-bundle formula and using that the restriction of $\roi(1)$ via $j$ and $\bar{s}_\infty$ is trivial.
\item Let $\epsilon\colon \Sm_{S,\et}\to \Sm_{S,\mathrm{Nis}}$ the canonical morphism of sites and let $E(\cal T(n))$ be a Nisnevich local fibrant spectrum corresponding to $R\epsilon_*\cal T(n)$. Then by (ii) $E(\cal T(n))$ satisfies $\textbf{SUB2}$.
\end{enumerate} 
\end{example}

{Before we state the next theorem, we make a notational remark: by abuse of notation, if $E$ is a presheaf and $A=\projlim A_i$ an ind-smooth scheme then we write $E(A)$ for $\indlim E(A_i)$. } 
\begin{theorem}
Let $S$ be the spectrum of a henselian discrete valuation ring with infinite residue field. Let $E\in \Spt_{S^1}(\Sm_S)$ be a Nisnevich local fibrant spectrum which satisfies $\textbf{SUB2}$ and $X\in \Sm_S$ of dimension $d$ with special fiber $Y$. Then the complex of sheaves over the small Nisnevich site $X_\mathrm{Nis}$
$$0\to (E^q_X)^\sim \to \bigoplus_{x\in X^{(0)}}\frak{z}_*\frak{z}^*E_{Z/X}^{q}\to \bigoplus_{x\in X^{(1)}}\frak{z}_*\frak{z}^*E_{Z/X}^{q+1}\to\dots\to \bigoplus_{x\in X^{(d)}}\frak{z}_*\frak{z}^*E_{Z/X}^{q+d}\to 0$$
is exact, except possibly at the spots $(E^q_X)^\sim$ and $\bigoplus_{x\in X^{(1)}}\frak{z}_*\frak{z}^*E_{Z/X}^{q+1}$. If for each $x\in X$ the forget support map for the special fiber 
$$E_{Y/X}(X_x^h)\to E_X(X_x^h)$$
is trivial, then it is exact everywhere. Furthermore, the same conclusion holds if for each $x\in Y$ the map $$E^q_X((X_x^h)_\eta) \to E^q_X(K(X_x^h))$$
is injective.
\end{theorem}
\begin{proof}
We begin with a preliminary paragraph. Let $V\in \Sm_S$ and $Z\hookrightarrow \bb A^1_V$ a closed subscheme such that the composition $Z\hookrightarrow \bb A^1_V\xto{\pi}V$ is finite. Let ${\bar{Z}}:=\pi(Z)_{\mathrm{red}}$ and
$$
\xymatrix{
\bb A^1_V \ar[d]_j \ar[r]^\pi & V \\
\bb P^1_V \ar[ru]_{\bar{\pi}}          & }
$$ be the factorisation of $\pi$ through the canonical open immersion $j$. Then we consider the diagram 
$$\xymatrix{
\bar{\pi}_*j_*E_{Z/\bb A^1_V}  \ar[r]   & \bar{\pi}_*j_*E_{\bb A^1_{\bar{Z}}/\bb A^1_V} & & & \\
 &  \ar[u] & E_{\bar{Z}/V} \ar[ul]_{\eta_\pi} & & \\
\bar{\pi}_*E_{Z/\bb P^1_V} \ar[r] \ar[uu]_{\bar{\pi}_*\eta_j}^\simeq  & \bar{\pi}_*E_{\bb P^1_{\bar{Z}}/\bb P^1_V} \ar[uu]^{^{\bar{\pi}_*\eta_j}} \ar[rr]^-{^{\bar{\pi}_*\eta_{j''}}} \ar[ur]^{\bar{\pi}_*\eta_{\bar{s}_\infty}} & & \bar{\pi}_*j''_*E_{(\bb P^1_{\bar{Z}}\setminus Z)/(\bb P^1_V\setminus Z)} \ar[ul]_{^{\bar{\pi}_*j''_*\eta_{s''_\infty}}} 
}$$
which commutates in the homotopy category.
The commutativity of the square on the left-handside follows by definition, the triangle in the middle commutes since $E$ satisfies $\textbf{SUB2}$, and the triangle on the right commutes by \cite[Lem. 5.4]{SS19}. The composition of the two lower horizontal maps is trivial since it is $\bar{\pi}_*$ applied to a localisation exact triangle (\cite[Lem. 3.7]{SS19}) and therefore the upper horizontal map $${\pi}_*E_{Z/\bb A^1_V}  \to {\pi}_*E_{\bb A^1_{\bar{Z}}/\bb A^1_V}$$ is trivial in the homotopy category.

Let $x\in Z\in Z^{p+1}(X)$ be a closed subset not containing the special fiber. Then by the geometric presentation lemma of Schmidt-Strunk \cite[Thm. 2.1]{SS18}, up to shrinking $X$ Nisnevich-locally around $x$, there exists a Nisnevich distinguished square 
$$
\xymatrix{
X\setminus Z \ar[d] \ar[r] & X\ar[d]^f\\
\bb A^1_V\setminus f(Z) \ar[r]          & \bb A^1_V}
$$
where $V\in \Sm_S$, $p:X\xto{f}\bb A^1_V\xto{\pi} V$ is a relative curve and such that 
$$Z\hookrightarrow X\xto{f}\bb A^1_V\xto{\pi} V$$
is finite. Let ${\bar{Z}}:=p(Z)_{\mathrm{red}}$ and $Z':=p^{-1}({\bar{Z}})\in Z^p(X)$. Since $E$ is Nisnevich local fibrant, by excision \cite[Lem. 3.11]{SS19} the upper horizontal morphism of the diagram
$$
\xymatrix{
E_{f(Z)/\bb A^1_V} \ar[d] \ar[r]^\simeq & f_*E_{Z/X} \ar[d]\\
E_{\bb A^1_{\bar{Z}}/\bb A^1_V} \ar[r]          & f_*E_{Z'/X}}
$$
is an equivalence. The vertical maps are the forget support maps. Applying $\pi_*$ gives the commutative diagram
$$
\xymatrix{
\pi_*E_{f(Z)/\bb A^1_V} \ar[d]_0 \ar[r]^\simeq & p_*E_{Z/X} \ar[d]\\
\pi_*E_{\bb A^1_{{\bar{Z}}}/\bb A^1_V} \ar[r]          & p_*E_{Z'/X}}
$$
in which the left horizontal morphism is zero by the discussion at the beginning of the proof. Therefore the forget support map
$$p_*E_{Z/X}\to p_*E_{Z'/X}$$
induces the trivial morphism in the homotopy category and in particular $E_{Z/X}(X)\to E_{Z'/X}(X)$ is trivial. By (\ref{eqn.exactness2}), this implies the exactness of the Gersten complex everywhere except at the spots $(E^q_X)^\sim$ and $\bigoplus_{x\in X^{(1)}}\frak{z}_*\frak{z}^*E_{Z/X}^{q+1}$. The reduction using (\ref{eqn.exactness1}) of the exactness at the other places to the forget support map $E_{Y/X}(X_x^h)\to E_X(X_x^h)$ being trivial now follows exactly as at the end of the proof of \cite[Thm. 5.13]{SS19}.\footnote{{It may be confusing that Schmidt-Strunk use a stronger hypothesis, i.e. triviality of a map of spectra in the homotopy category and not of homotopy groups to deduce that maps on homotopy groups are zero. However, they actually deduce a stronger statement, i.e. that $E_{Z/X}(X_x^h)\to E_{X}(X_x^h)$ is trivial; see \cite[End of Proof of 5.13]{SS19}. This explains why our weaker assumption on homotopy groups actually suffices.}} Note, however, that $X_x^h$ is not necessarily a discrete valuation ring. To reduce further {and to prove the last statement of the theorem}, we may assume that $x\in Y$ and consider the commutative diagram of exact {sequences}
$$
\xymatrix{
(E^q_{X^{(1)}})_x \ar[d]^\cong \ar[r] & (E^q_{X^{(0)}})_x \ar[d]^\cong \ar[r] & (E^q_{X^{(0/1)}})_x \ar[d]^\cong  \\
(E^q_{X^{(1)}})_x  \ar[r] & E^q_X(X_x^h) \ar[r] & E^q_X(K(X_x^h)).
}
$$
The right vertical isomorphisms exists by \cite[Cor. 3.21]{SS19}, the other two by definition. It therefore suffices to show that the lower right horizontal map is injective. Let $\eta$ be the generic point if $Y$. Then it factorises as $E^q_X(X_x^h) \to E^q_X((X_x^h)_\eta) \to E^q_X(K(X_x^h))$ and the first map is obtained from the second map in the exact {sequence}
$$E^q_{X^{(1)}-\eta}\to E^q_{X^{(0)}}\to E^q_{X^{(0)}/X^{(1)}-\eta}$$
by localising at $x$ since $E^q_{X^{(0)}/X^{(1)}-\eta}(X_x^h)\cong E^q_X((X_x^h)_\eta)$. By what we have shown above, the first map in this exact {sequence} is {zero on stalks}. 
\end{proof}

% there exists a functorial norm map $N_{B'/B}:\cal F(B')\to \cal F(B)$ for any finite \'etale extension of local rings $B\to B'$ such that if $B\to B'$ is of constant degree $d$, then the composition $\cal F(B)\to\cal F(B')\xto{N_{B'/B}} \cal F(B')$ is multiplication by $d$,

%(see \cite[Lem. 6.11]{ElmantoMorrow} in equal characteristic) and has a Gersten resolution (for more details on the Gersten resolution for spectra see \cite[Sec. 5.2]{CHK97} and \cite{SS19}) in the Nisnevich topology.

\section{The $p$-adic cycle class map} 
We fix the following notation. Let $A$ be a discrete valuation ring with local parameter $\pi$, perfect residue field $k$ of characteristic $p>0$ and fraction field $K$ of characteristic zero. Let $S=\Spec(A)$. Let $X=\Spec(B)$ be an ind-smooth local $S$-scheme and assume that the structure morphism is surjective. Let $I=(\pi)$ be the ideal defining the special fiber. We denote by $B^h$ the henselisation of $B$ along $I$. Let $i:\Spec B/I\to \Spec B^h$ be the inclusion. In the entire section let $n\geq 0$.
%Let $(B,\mathfrak{m}_B)$ be a henselian local ring of mixed characteristic $(0,p)$ which is essentially smooth over a discrete valuation ring $(A,\mathfrak{m}_A=(\pi))$ and $(p)\subset I\subset \mathfrak{m}_B$ an ideal. 
 
%By Gabber's affine analogue of proper base change \cite[Thm. 1]{GabberAffineAnalogue},\footnote{If $I=\mathfrak{m}_B$ then this is easier since then the map $B^h\to B/I$ induces an equivalence between the categories of finite \'etale $B^h$-algebras and finite \'etale $B/I$-algebras and any \'etale cover of $B^h$ has a refinement by a finite \'etale cover. The statement then follows from \cite[III.3.3]{Milne1986}.} the restriction map
%$$H^{s}(\Spec B^h,\T_r(n))\xto\cong H^{s}(\Spec B/I,i^*\T_r(n))$$
%is an isomorphism for all $s$.
We need the following theorem due to  Sakagaito (see also \cite[Sec. 5]{AMMN}).

\begin{theorem}\label{theorem_top_cohomology} (\cite[Thm. 3.7]{sakagaito2023etale})
Let the notation be as above. Then $R\epsilon_* i^*\T_r(n)\in D((\Spec B/I)_{\rm Nis},\bb Z/p^r)$ is concentrated in degree $[0,n+1]$. Furthermore,
$$H^{n+1}(\Spec B^h,\T_r(n))\cong H^{1}(\Spec B/I,W_r\Omega^n_{\log}).$$
\end{theorem}

\begin{theorem}\label{theorem_comparison}
%Let $A$ be a discrete valuation ring with local parameter $\pi$, perfect residue field $k$ of characteristic $p>0$ and fraction field $K$ of characteristic zero. 
Assume that $k$ is finite. Let $X$ be a smooth projective $S$-scheme with special fiber $X_k$ (in the following also denoted by $X_1$). We denote the inclusion of the special fiber by $i:X_k\to X$ and set $X_n=X\times_A A/\pi^n$. Let $d$ be the relative dimension of $X$ over $S$.
Then for $q\in\{0,1\}$ there are isomorphisms
$$H^{d-q}_{\rm Nis}(X_k,i^*\hat{\cal K}^M_{d,X}/p^r) \xto{\cong}  \projlim_n H^{d-q}_{\rm Nis}(X_k,\hat{\cal K}^M_{d,X_n}/p^r)\xto{\cong} H^{2d-q}_\et(X_k,i^*\cal T_r(d)). $$
\end{theorem}
\begin{proof}
%The Galois symbol map $\hat{\cal K}^M_{d,X}/p^r\to \cal H^d(\cal T_r(d)_X)$ induces a map $cl_X:H^d(X,\hat{\cal K}^M_{d,X}/p^r)\to H^{2d}(X,\cal T_r(d))$. 
%Let $R\epsilon_*\cal T_r(d)_{X}:=\cal T_r(d)_{X,\rm Nis}$.
In the following we denote the change of sites $X_\et\to X_{\rm Nis}$ as well as the change of sites $({X_k})_\et\to ({X_k})_{\rm Nis}$ by $\epsilon$. We first note that the Galois symbol map $i^*\hat{\cal K}^M_{d,X,\rm Nis}/p^r\to i^*\cal H^d(\cal T_r(d)_{X,\rm Nis})$ factorises through $\K^M_{d,X_n,\rm Nis}/p^r$ for $n$ large enough. Via a norm argument one first reduces this statement to the case in which the residue field is large and therefore the improved Milnor $K$-theory coincides with ordinary Milnor $K$-theory. This case can be deduced from the short exact sequence $0\to 1+\pi^n\roi_X \to \roi_X^\times \to \roi_{X_n}^\times \to 0$, which induces the unit filtration on Milnor $K$-theory, noticing that in the Nisnevich topology the group $(1+\pi^n\roi_X)/p^r=0$ for $n$ large enough by Hensel's lemma \cite[II, Lem. 2]{Elkik1973}. 
Indeed, the kernel of the restriction map $i^*\cal K^M_{d,X}\to \cal K^M_{d,X_n}$ is locally generated by elements of the form $\{1+\pi^n a,f_1,\dots,f_{d-1}\}$ with $f_1,\dots,f_{d-1}$ units \cite[Lemma 1.3.1]{KaS86}. This implies the first isomorphism of the theorem (without any restriction on $q$).
Next we consider the Leray spectral sequence
$$H^p_{\rm Nis}(X_k,R^q\epsilon_*(i^*\cal T_r(d)_X))\Longrightarrow H^{p+q}_\et(X_k,i^*\cal T_r(d)).$$
By \cite[Prop. 2.2(b)]{Geisser2004} we have that $R^q\epsilon_*(i^*\cal T_r(d)_X)\cong i^*R^q\epsilon_*(\cal T_r(d)_{X})$.
By \cite{LuMo20} the map $i^*\hat{\cal K}^M_{d,X,\rm Nis}/p^r\to  i^*R^d\epsilon_*(\cal T_r(d)_{X})$ is an isomorphism and by Theorem \ref{theorem_top_cohomology} we know that $R^{d+1}\epsilon_*(i^*\cal T_r(d)_X)= R^1\epsilon_*W_r\Omega^d_{\log}$ and $R^q\epsilon_*(i^*\cal T_r(d)_X)=0$ for $q>d+1$. 
This implies the existence of an exact sequence %\morten{Ist wirklich $H^{d-1}_{\rm Nis}(X_k,i^*\hat{\cal K}^M_{d,X}/p^r)\to H^{2d-1}_{\rm Nis}(X_k,i^*\cal T_r(d))$?}
$$H^{d-3}_{\rm Nis}(X_k,R^1\epsilon_*W_r\Omega^d_{\log})\to H^{d-1}_{\rm Nis}(X_k,i^*\hat{\cal K}^M_{d,X}/p^r)\to H^{2d-1}_\et(X_k,i^*\cal T_r(d))\to $$
$$H^{d-2}_{\rm Nis}(X_k,R^1\epsilon_*W_r\Omega^d_{\log})\to H^{d}_{\rm Nis}(X_k,i^*\hat{\cal K}^M_{d,X}/p^r)\to H^{2d}_\et(X_k,i^*\cal T_r(d))\to $$
$$H^{d-1}_{\rm Nis}(X_k,R^1\epsilon_*W_r\Omega^d_{\log})\to 0\to H^{2d+1}_\et(X_k,i^*\cal T_r(d)) \to  $$
$$H^{d}_{\rm Nis}(X_k,R^1\epsilon_*W_r\Omega^d_{\log})\to 0.$$
For the exactness of the first row note that the Gersten resolution implies that $H^{d}_{\rm Nis}(X_k,R^{d-1}\epsilon_*(i^*\cal T_r(d)_X))=0$. The Kato conjectures say that $H^{q}_{\rm Nis}(X_k,R^1\epsilon_*W_r\Omega^d_{\log})=0$ for $q<d$ and isomorphic to $\bb Z/p^r$ for $q=d$. By \cite[Thm. 0.3]{JS} this holds for $d-4\leq q \leq d$ which implies the theorem.
\end{proof}

Let $q\in\{0,1\}$. Theorem \ref{theorem_comparison} implies that there is a commutative diagram
$$
\xymatrix{
H^{d-q}_{\rm Nis}(X,\hat{\cal K}^M_{d,X}/p^r)\ar[r]^-{res^d_{/p^r}} \ar[d]_{cl_X} & \projlim_n H^{d-q}_{\rm Nis}(X_k,\K^M_{d,X_n}/p^r) \ar[d]^{cl_{X_k}}  \\
H^{2d-q}_\et(X, \cal T_r(d)) \ar[r]^-\cong & H^{2d-q}_\et(X_k,i^*\cal T_r(d))   \\
}
$$
in which the bottom map is an isomorphism by \'etale proper base change and $cl_{X_k}$ is an isomorphism by Theorem \ref{theorem_comparison}. 
By \cite[Thm. 0.2]{LuMo20} there is an isomorphism 
$\CH^d(X)/p^r\cong H^{d}_{\rm Nis}(X,\hat{\cal K}^M_{d,X}/p^r)$.
By \cite[Proof of Cor. 2.7]{Lu16} and the Gersten conjecture for $\hat{\cal K}^M_{d,X}/p^r$ \cite[Thm. 0.2]{LuMo20}, there is an isomorphism
$\CH^d(X,1,\bb Z/p^r)\cong H^{d-1}_{\rm Nis}(X,\hat{\cal K}^M_{d,X}/p^r).$
It is an interesting question to ask if the maps $res^d_{/p^r}$ and $cl_X$ are isomorphisms (see \cite{Lueders2019}, \cite{SS14} for the surjectivity and \cite[Conj. 10.1]{KerzEsnaultWittenberg2016} for the original question in the case $q=0$).

%\section*{Conflict of interest and data}
%The author is support by the Deutsche Forschungsgemeinschaft (DFG, German Research Foundation) through the Collaborative Research Centre TRR 326 \textit{Geometry and Arithmetic of Uniformized Structures}, project number 444845124. Except for that, the corresponding author states that there is no conflict of interest. The manuscript contains no associated data and just uses the references listed below.

\bibliographystyle{acm}
\bibliography{Bibliografie} 

\noindent
\parbox{0.5\linewidth}{
\noindent
Morten L\"uders \\
Universität Heidelberg\\
Institut für Mathematik \\
Im Neuenheimer Feld 205 \\
69120 Heidelberg \\
Germany\\
{\tt mlueders@mathi.uni-heidelberg.de}
}
\end{document}